\documentclass[12pt,reqno]{amsart}
\usepackage{amssymb, amsmath, amsthm, amsfonts,xcolor, enumerate}
\usepackage{verbatim}
\usepackage{hyperref}

\usepackage{amsmath, amsopn, amssymb, amsthm}
\usepackage{enumerate}
\usepackage{mathtools}

\newtheorem{thm}{Theorem}[section]

\newtheorem{lemma}[thm]{Lemma}
\newtheorem{cor}[thm]{Corollary}
\newtheorem{prop}[thm]{Proposition}
\newtheorem{obs}[thm]{Observation}

\theoremstyle{definition}
\newtheorem{defn}[thm]{Definition}
\newtheorem{rmk}[thm]{Remark}
\newtheorem*{SSalg}{The Schramm-Steif randomized algorithm}

\def\Pr{\mathbb{P}}
\def\Ex{\mathbb{E}}
\def\Inf{\textup{Inf}}
\def\Var{\textup{Var}}
\def\Cov{\textup{Cov}}
\newcommand{\N}{\mathbb{N}}
\newcommand{\Z}{\mathbb{Z}}
\newcommand{\RR}{\mathbb{R}}
\newcommand{\HH}{\mathbb{H}}
\def\AA{\mathbb{A}}
\newcommand{\A}{\mathcal{A}}
\newcommand{\C}{\mathcal{C}}

\newcommand{\F}{\mathcal{F}}

\def\le{\leqslant}
\def\ge{\geqslant}
\def\eps{\varepsilon}
\def\ds{\displaystyle}
\def\ol{\overline}

\title{Quenched Voronoi percolation}

\author{Daniel Ahlberg}
\address{IMPA, Estrada Dona Castorina 110, Jardim Bot\^anico, Rio de Janeiro, RJ, Brasil \and Department of Mathematics, Uppsala University, SE-75106 Uppsala, Sweden} \email{ahlberg@impa.br}

\author{Simon Griffiths}
\address{Department of Statistics, University of Oxford, Oxford, United Kingdom} \email{simon.griffiths@stats.ox.ac.uk}	

\author{Robert Morris}
\address{IMPA, Estrada Dona Castorina 110, Jardim Bot\^anico, Rio de Janeiro, RJ, Brasil} \email{rob@impa.br}

\author{Vincent Tassion}
\address{D\'epartement de Math\'ematiques Universit\'e de Gen\`eve, Gen\`eve, Switzerland} \email{Vincent.Tassion@unige.ch}

\thanks{Research supported in part by postdoctoral grant 637-2013-7302 from the Swedish Research Council (DA), EPSRC grant EP/J019496/1 (SG), a CNPq bolsa de Produtividade em Pesquisa (RM), and ANR grant MAC2 (ANR-10-BLAN-0123) and the Swiss NSF (VT)}

\begin{document}

\begin{abstract}
We prove that the probability of crossing a large square in quenched Voronoi percolation converges to $1/2$ at criticality, confirming a conjecture of Benjamini, Kalai and Schramm from 1999. The main new tools are a quenched version of the box-crossing property for Voronoi percolation at criticality, and an Efron-Stein type bound on the variance of the probability of the crossing event in terms of the sum of the squares of the influences. As a corollary of the proof, we moreover obtain that the quenched crossing event at criticality is almost surely noise sensitive.
\end{abstract}

\maketitle

\section{Introduction}\label{sec:intro}

The \emph{noise sensitivity} of a Boolean function was introduced in 1999 in a seminal paper of Benjamini, Kalai and Schramm~\cite{BKS}, and has since developed into an important area of probability theory (see, e.g.,~\cite{GPS,GS,SS}), linking discrete Fourier analysis with percolation theory and combinatorics. One of the main results of~\cite{BKS} gave a sufficient condition for a sequence of functions $f_n \colon \{0,1\}^n \to \{0,1\}$ to be sensitive to small amounts of random noise in the following precise sense: if $\omega \in \{0,1\}^n$ is chosen uniformly at random, and $\omega^\eps$ is obtained from $\omega$ by resampling each variable with some fixed probability $\eps > 0$, then $f_n(\omega)$ and $f_n(\omega^\eps)$ are asymptotically independent. They used this theorem to show that the sequence of functions which encodes crossings of $n \times n$ squares in bond percolation on $\Z^2$ is noise sensitive. Thus, even if one knows all but a \emph{random} $o(1)$-proportion of the edges, one still (with high probability) has very little information about the crossing event.  

The authors of~\cite{BKS} furthermore made a number of conjectures regarding more precise notions of sensitivity and sensitivity to different types of noise. Several of these conjectures have since played an important role in the subsequent development of the area, most spectacularly in~\cite{SS} and~\cite{GPS}, where extremely precise results were obtained about the Fourier spectrum of the crossing event, and about the `dynamical percolation' process introduced by H\"aggstr\"om, Peres and Steif~\cite{HPS} and (independently) by Benjamini, see~\cite{S}. To give another example,  they made the following conjecture for Bernoulli bond percolation on the square lattice: even if you are told the status of \emph{all} the vertical edges, you still have very little information about the crossing event. This conjecture was proved by Garban, Pete and Schramm~\cite[Theorem~1.3]{GPS}, as a consequence of their very precise bounds on the Fourier spectrum. Note that in this theorem we are given a \emph{deterministic} set of edges (of density $1/2$), rather than a \emph{random} set of edges (of density $1 - o(1)$) as in the result stated above.

In this paper, we will prove a similar result (also conjectured in~\cite{BKS}) in the setting of \emph{Voronoi percolation}: that knowing the point set (but not the colours of the cells) gives asymptotically no information about the crossing event. In order to state our main result precisely, we will need a few basic definitions.

Consider a set $\eta$ of $n$ points in the square $S = [0,1]^2$, each chosen independently and uniformly at random. For each $u \in \eta$, define the Voronoi (or Dirichlet) cell\footnote{The study of these objects dates back at least to Dirichlet~\cite{D} in 1850, who used them in his work on quadratic forms, although they appear to have been introduced even earlier, by Kepler and (independently) Descartes, see~\cite{LP}. The natural generalisation to $d$ dimensions was first studied by Voronoi~\cite{V} in 1908.} of $u$ to be
$$C(u) \, = \, \big\{ x \in [0,1]^2 \,:\, \| u - x \|_2 \le \| v - x \|_2 \textup{ for every } v \in \eta \big\},$$
and let $\omega \colon \eta \to \{-1,1\}$ be a uniformly random two-colouring of the points of $\eta$; we will call the points $u$ (and the associated cells $C(u)$) with $\omega(u) = 1$ `red' and those with $\omega(u) = -1$ `blue'. We say that there is a \emph{red horizontal crossing} of $S$ if there is a path from the left- to the right-hand side of $S$ that only intersects red cells, and write $H_S$ for the event that there exists such a red horizontal crossing of $S$. Note that $\Pr( H_S ) = 1/2$, by symmetry. We refer the reader who is unfamiliar with Voronoi percolation to~\cite{BRbook} for a more extensive introduction.

The following theorem confirms (in a strong form) a conjecture of Benjamini, Kalai and Schramm~\cite{BKS}.

\begin{thm}\label{thm:BKSconj}
There exists $c > 0$ such that
$$\Pr\left( \frac{1}{2} - n^{-c} \,\le\, \Pr\big( H_S \,|\, \eta \big) \,\le\, \frac{1}{2} + n^{-c} \right) \, \ge \, 1 - n^{-c}$$
for all sufficiently large $n \in \N$.
\end{thm}

Let $f^\eta \colon \{-1,1\}^\eta \to \{0,1\}$ be the function such that $f^\eta(\omega) = 1$ if and only if $H_S$ holds. The key new idea of the proof of Theorem~\ref{thm:BKSconj} is 
the following Efron-Steif type bound (see Theorem~\ref{thm:VarInf}, below) on the variance of the probability of the crossing event in terms of the influences of $f^\eta$, which can be viewed as a random Boolean function:
\begin{equation}\label{eq:varinf}
\Var\Big( \Pr\big( H_S \,|\, \eta \big) \Big) \, \le \, \sum_{m=1}^n \Ex\big[ \Inf_m( f^\eta )^2 \big].
\end{equation}
Recall that the influence $\Inf_m(f_n)$ of the $m^{th}$ variable of a Boolean function $f_n \colon \{-1,1\}^n \to \{0,1\}$ is defined to be the expected absolute change in $f_n$ when the sign of the $m^{th}$ variable is flipped, i.e.,
$$\Inf_m(f_n) \, = \, \Pr\big( f_n(\omega) \ne f_n(\omega') \big),$$
where $\omega$ is chosen uniformly, and $\omega'$ is obtained from $\omega$ by flipping the $m^{th}$ variable. Benjamini, Kalai and Schramm~\cite{BKS} proved that
\begin{equation}\label{eq:BKSthm}
  \sum_{m = 1}^n \Inf_m(f_n)^2 \to 0 \; \textup{ as } n \to \infty \qquad \Rightarrow \qquad (f_n)_{n \in \N} \textup{ is noise sensitive,}
\end{equation}
and moreover introduced a technique (the `algorithm method', see below) which can often be used to bound $\sum_{m = 1}^n \Inf_m(f_n)^2$ when $f_n$ encodes crossing events in percolation models. We will use this method (or, more precisely, the `randomized' version of it developed by Schramm and Steif~\cite{SS}), together with a new `box-crossing property' for quenched Voronoi percolation (see below), to bound\footnote{More precisely, since $f^\eta$ is a random function we will prove that our bound on $\sum_{m = 1}^n \Inf_m(f^\eta)^2$  holds with high probability as $n \to \infty$.} $\sum_{m = 1}^n \Inf_m(f^\eta)^2$, and hence deduce Theorem~\ref{thm:BKSconj}.

As an immediate consequence of the proof outlined above, together with~\eqref{eq:BKSthm}, we also obtain the following theorem. Let us say that \emph{quenched Voronoi percolation is almost surely noise sensitive at criticality} if
\begin{equation}\label{eq:NSdef}
\Ex\big[ f^\eta(\omega) f^\eta(\omega^\eps) \,|\, \eta \big] - \Ex\big[ f^\eta(\omega) \,|\, \eta \big]  \Ex\big[ f^\eta(\omega^\eps) \,|\, \eta \big] \to 0
\end{equation}
as $n \to \infty$ with probability 1 for every $\eps \in (0,1)$, where $\omega$ and $\omega^\eps$ are as defined above.

\begin{thm}\label{thm:QV:NS}
Quenched Voronoi percolation is almost surely noise sensitive at criticality.
\end{thm}

In fact, as a consequence of the Schramm-Steif method, we obtain a stronger result: that the noise sensitivity exponent for quenched Voronoi percolation is positive. This means that there exists a constant $c > 0$ such that~\eqref{eq:NSdef} holds even if $\eps = n^{-c}$.  

\begin{rmk}
The word `quenched' refers to the fact that we are proving a statement which holds for almost all choices of $\eta$. The phrase `at criticality' refers to the fact that $\omega$ is chosen uniformly at random. We remind the reader that the critical probability of Voronoi percolation in the plane is $1/2$, as was proved by Bollob\'as and Riordan~\cite{BR}.
\end{rmk}

We remark that Theorem~\ref{thm:QV:NS} is not the first result of this type for a continuum percolation model. Indeed, a similar theorem for the Poisson Boolean model\footnote{In this model, a pair $u,v \in \eta$ is considered to be adjacent if the distance between them is at most 1.} was proved by the first three authors with Broman~\cite{ABGM}, and the techniques introduced in that paper have recently been extended by the first three authors with Balister and Bollob\'as~\cite{ABBGM} to the settings of (annealed) Voronoi percolation and the Poisson Boolean model with random radii. (In each case the point set $\eta$ is perturbed, together with the colours/radii.) We emphasize, however, that the techniques introduced in this paper are completely different from those used in~\cite{ABBGM,ABGM}, where the method involved choosing the point set in two stages, and applying the algorithm method in the non-uniform setting. Indeed, none of the previously-introduced techniques seem to have any chance of working in the setting of quenched Voronoi percolation.

As mentioned above, in order to use the algorithm method we will need to prove a 1-arm estimate that will follow from a quenched version of the box-crossing property for Voronoi percolation at criticality. This result gives bounds on the probability that a rectangle (of fixed aspect ratio) is crossed at criticality, and is an analogue of the celebrated results for bond percolation on $\Z^2$ of Russo~\cite{R} and Seymour and Welsh~\cite{SW}. Corresponding results have been obtained in various related settings, and obtaining such bounds is frequently a key step in the proof of various important applications, see e.g.~\cite{A,BR,DHN,LS,Roy,T}. In particular, an important breakthrough was made by Bollob\'as and Riordan~\cite{BR}, who proved an RSW-type theorem for (annealed) Voronoi percolation\footnote{More precisely, they proved that there exists an infinite sequence of values of $L$ such that the probability of crossing an $L \times \lambda L$ rectangle is bounded away from 0.}, and used it to deduce that the critical probability for percolation is $1/2$. The full box-crossing property in the annealed setting was obtained only very recently, by the fourth author~\cite{T}. We remark that this result will play an important role in our proof of Theorem~\ref{thm:RSW}, below. 

As above, we write $H_R$ for the event that there is a red horizontal crossing of $R$.

\begin{thm}[The quenched box-crossing property for Voronoi percolation]\label{thm:RSW}
For every rectangle $R \subset \RR^2$, there exists a constant $c > 0$ such that the following holds. Let $n \in \N$, let $\eta \subset R$ be a set of $n$ points, each chosen uniformly at random, and let $\omega \colon \eta \to \{-1,1\}$ be a uniform colouring. Then
$$\Pr\Big( c < \Pr\big( H_R \,|\, \eta \big) < 1 - c \Big) \to \, 1$$
as $n \to \infty$. 
\end{thm}

We remark that, moreover, for every $\gamma > 0$ there exists $c = c(\gamma,R) > 0$ such that $\Pr\big( H_R \,|\, \eta \big) \not\in (c,1-c)$ has probability at most $n^{-\gamma}$. An analogous theorem if $\eta$ is a Poisson point process in the plane (or in the half-plane) follows by exactly the same proof. 

We will prove Theorem~\ref{thm:RSW} in three steps. First, we will prove a weaker result for Voronoi percolation in the plane (see Theorem~\ref{thm:RSW:plane}): this says that there exists a
constant $c > 0$ such that
\begin{equation}\label{eq:weakRSW}
  \Pr\bigg( \Pr\big( H_R \,|\, \eta \big) \le \frac{1}{2^k} \bigg) \le \, (1 - c)^k
\end{equation}
for all sufficiently large $k$. We will then deduce an analogous statement for Voronoi percolation in a half-plane; somewhat surprisingly, the deduction is not trivial, and we will have to do some work to deal with the boundary effects (see Section~\ref{sec:RSW:halfplane}). Finally, we will use these results, together with the algorithm method (see Section~\ref{sec:proof}) and our Efron-Stein type inequality~\eqref{eq:varinf}, proved in Section~\ref{sec:VarInf}, to show (see Theorem~\ref{thm:BKS:rectangle}) that 
$$\Pr\big( H_R \,|\, \eta \big) \, \to \, \Ex\big[ \Pr\big( H_R \,|\, \eta \big) \big]$$
in probability, as $n \to \infty$. This result will imply both Theorem~\ref{thm:BKSconj} and Theorem~\ref{thm:RSW}, using the box-crossing property for annealed Voronoi proved in~\cite{T}.

The organisation of the rest of the paper is as follows. First, in Section~\ref{sec:VarInf}, we will bound the variance of the probability of the crossing event by the expected sum of the squares of the influences of $f^\eta$. We will do so by introducing a martingale, whose steps correspond to choosing the points of $\eta$ one-by-one, and bounding the variance of step $m$ in terms of the expectation of the square of the influence of the $m^{th}$ element of $\eta$, see Lemma~\ref{lem:key:VarInf}. Armed with this lemma, the claimed bound~\eqref{eq:varinf} follows easily.

Second, in Section~\ref{sec:RSW}, we will prove weak bounds for the crossing probabilities in quenched Voronoi percolation~\eqref{eq:weakRSW} in both the plane, and the half-plane. The key tools in our (surprisingly simple) proof will be the `box-crossing property' for annealed Voronoi percolation, proved in~\cite{T}, together with colour-switching. In particular, we would like to highlight Lemma~\ref{magic:lemma}, which states that 
$$\Pr\big( H_R \,|\, \eta \big) \, = \, \Ex\big[ 2^{-X} \,|\, \eta \big],$$
where $X$ is the random variable which counts the number of vertex-disjoint vertical monochromatic crossings of $R$. Although this lemma, once stated, is easy to prove, we have found it to be extremely useful, and expect it to have many other applications. 

Finally, in Section~\ref{sec:proof}, we will complete the proof of the main theorems, by using the algorithm method of Benjamini, Kalai and Schramm~\cite{BKS} and Schramm and Steif~\cite{SS} to bound the sum of the squares of the influences of $f^\eta$. (Indeed, once we are armed with the results of Section~\ref{sec:RSW}, the required bound follows by simply repeating the method of~\cite{SS}.) Combining this bound with the results of Section~\ref{sec:VarInf}, we obtain Theorem~\ref{thm:BKSconj}. By~\eqref{eq:BKSthm}, we obtain Theorem~\ref{thm:QV:NS}, and by the box-crossing property for annealed Voronoi percolation~\cite{BR,T}, we obtain Theorem~\ref{thm:RSW}.

\section{Variance and influence}\label{sec:VarInf}

In this section, we will prove a somewhat surprising bound on the (typical) dependence of the crossing event on the point set $\eta$ in terms of the (expected) influences of the colours. Since we will need to use the results of this section in the proof of Theorem~\ref{thm:RSW}, as well as that of Theorem~\ref{thm:BKSconj}, we will work in the more general set-up of an arbitrary rectangle $R \subset \RR^2$, so let $\eta$ be a set of $n$ points in $R$, each of which is chosen independently and uniformly at random. We will write $f_R^\eta \colon \{-1,1\}^\eta \to \{0,1\}$ for the function that encodes whether or not there is a red horizontal crossing of $R$ in the corresponding Voronoi tiling. Recall that 
$$\Inf_m(f_R^\eta) \, := \, \Pr\left( f_R^\eta(\omega) \neq f_R^\eta(\omega') \,\big|\, \eta \right),$$ 
where $\omega'$ equals $\omega$ except on the $m^{th}$ element of $\eta$. 

\pagebreak

The main result of this section is the following inequality, which is highly reminiscent of the well-known inequality of Efron and Stein~\cite{ES}. 

\begin{thm}\label{thm:VarInf}
For every rectangle $R \subset \RR^2$, 
$$\Var\Big( \Pr\big( H_R \,|\, \eta \big) \Big) \, \le \, \sum_{m=1}^n \Ex\big[ \Inf_m( f_R^\eta )^2 \big].$$
\end{thm}
 
Note that the following corollary  is an immediate consequence of the above theorem and Chebychev's inequality.
  
\begin{cor}\label{cor:VarInf}
Let $a(n) = \Ex\Big[ \sum_{m = 1}^n \Inf_m(f_R^\eta)^2 \Big]$. Then 
$$\Pr\Big( \big| \Pr( H_R \,|\, \eta ) - \Pr( H_R ) \big| \ge a(n)^{1/3} \Big) \, \le \, a(n)^{1/3}.$$
\end{cor}

The proof of Theorem~\ref{thm:VarInf} uses the following simple martingale $(q_m)_{m=0}^{n}$. Let us choose the elements of $\eta$ one-by-one, and let $\eta_m$ denote the $m^{th}$ element. Now write
$$q^\eta \, = \, \Pr\big( H_R \,\big|\, \eta \big)$$ 
for the probability of such a crossing given $\eta$, and define
$$q_m \, := \, \Ex\big[ q^\eta \,|\, \F_m \big],$$
where $\F_m$ denotes the $\sigma$-algebra generated by $\eta_1,\ldots,\eta_m$. 


\begin{obs}\label{obs:uncorrelated}
$\Var\big( q^\eta \big) = \, \ds\sum_{m=1}^n \Var\big( q_m - q_{m-1} \big).$
\end{obs}


\begin{proof}
Since $q^\eta - \Ex[q^\eta] = \sum_{m=1}^n (q_m - q_{m-1})$, it will suffice to show that $\Cov(q_i - q_{i-1},q_j - q_{j-1}) = 0$ for every $1 \le i < j \le n$. To see this, we simply condition on $\F_i$, which gives
$$\Cov(q_i - q_{i-1},q_j - q_{j-1}) \, = \, \Ex\Big[ \Ex\big[ \big( q_i - q_{i-1} \big)\big( q_j - q_{j-1} \big) \,\big|\, \F_i \big] \Big] \, = \, 0,$$
since $\Ex\big[ q_j - q_{j-1} \,|\, \F_i \big] = 0$, and $q_i - q_{i-1}$ is determined by $\eta_1,\ldots,\eta_i$. 
\end{proof}
 
By Observation~\ref{obs:uncorrelated}, the following lemma completes the proof of Theorem~\ref{thm:VarInf}. 

\begin{lemma}\label{lem:key:VarInf}
For every $1 \le m \le n$,
$$\Var\big( q_m - q_{m-1} \big) \, \le \, \Ex\big[ \Inf_m\big( f_R^\eta \big)^2 \big]$$
almost surely.
\end{lemma}

\begin{proof}
First observe that, since $\Ex\big[ q_m - q_{m-1} \,|\, \F_{m-1} \big] = 0$ almost surely, by the conditional variance formula\footnote{That is, $\Var(X) = \Var\big( \Ex[ X \,|\, Y ] \big) + \Ex\big[ \Var( X \,|\, Y ) \big]$.} we have
\begin{equation}\label{eq:var:stepone}
\Var\big( q_m - q_{m-1} \big) \, = \, \Ex\big[ \Var\big( q_m \,\big|\, \F_{m-1} \big) \big].
\end{equation}
Now, let $\eta^-$ be obtained from $\eta$ by deleting $\eta_m$. Since $q^{\eta^-}$ does not depend on the $m^{th}$ element of $\eta$, it follows that
\begin{align*}
\Var\big( q_m \,\big|\, \F_{m-1} \big) & \, = \, \Var\Big( \Ex\big[ q^\eta \,|\, \F_m \big] - \Ex\big[ q^{\eta^-} \,|\, \F_{m-1} \big]  \,\big|\, \F_{m-1} \Big)\\
& \, = \, \Var\Big( \Ex\big[ q^\eta - q^{\eta^-} \,|\, \F_m \big] \,\big|\, \F_{m-1} \Big).
\end{align*}
Now, since $\Var(X) \le \Ex[X^2]$ for every random variable $X$, this is at most
$$\Ex\Big[ \Ex\big[ q^\eta - q^{\eta^-} \,|\, \F_m \big]^2  \,\big|\, \F_{m-1} \Big],$$
which is in turn at most
$$\Ex\Big[ \big( q^\eta - q^{\eta^-} \big)^2 \,\big|\, \F_{m-1} \Big],$$
by Jensen's inequality. We make the following claim, which completes the proof.

\medskip

\noindent \textbf{Claim:} $|q^\eta - q^{\eta^-}| \le \Inf_m(f_R^\eta)$, almost surely.

\begin{proof}[Proof of claim]
Let us write $\omega^{+}$ (resp. $\omega^{-}$) for the element of $\{-1,1\}^n$ obtained from $\omega$ by setting the $m^{th}$ coordinate equal to $1$ (resp. $-1$). Also define $f_R^{\eta^-}(\omega)$ for $\omega \in \{-1,1\}^n$ in the obvious way, by ignoring the $m^{th}$ coordinate of $\omega$. We have
\begin{align*}
q^\eta - q^{\eta^-} & \, \le \, \Pr\big( f_R^\eta(\omega) > f_R^{\eta^-}(\omega) \,|\, \eta \big)\\
& \, \le \, \Pr\big( f_R^\eta(\omega^{+}) > f_R^{\eta}(\omega^{-}) \,|\, \eta \big) \, = \, \Inf_m(f_R^\eta),
\end{align*}
since $f_R^\eta$ is monotone (as a function on $\{-1,0,1\}^\eta$).\footnote{Indeed, abusing notation slightly, we can define a function $f_R^\eta \colon \{-1,0,1\}^\eta \to \{0,1\}$ by setting $f_R^\eta(\omega) = 1$ if there is a red horizontal crossing in the Voronoi tiling defined by the set $\eta' = \{u \in \eta : \omega(u) \ne 0 \}$ with colouring $\omega' = \omega |_{\eta'}$. The function $f_R^\eta$ is monotone increasing since (from the point of view of the crossing event) a red point is better than no point, and no point is better than a blue point.} An identical calculation shows that $q^{\eta^-} - q^{\eta} \le \Inf_m(f_R^\eta)$, and so the claim follows.
\end{proof}

The lemma now follows since, by~\eqref{eq:var:stepone}, we have
\begin{align*}
\Var\big( q_m - q_{m-1} \big) & \, = \, \Ex\big[ \Var\big( q_m \,\big|\, \F_{m-1} \big) \big] \, \le \, \Ex\Big[ \Ex\big[ \big( q^\eta - q^{\eta^-} \big)^2 \,\big|\, \F_{m-1} \big] \Big] \\
& \, \le \, \Ex\Big[  \Ex\big[ \Inf_m\big( f_R^\eta \big)^2 \,\big|\, \F_{m-1} \big] \Big]  \, = \, \Ex\big[ \Inf_m\big( f_R^\eta \big)^2 \big],
\end{align*}
as required.
\end{proof}

As noted above, Theorem~\ref{thm:VarInf} follows immediately from  Observation~\ref{obs:uncorrelated} and Lemma~\ref{lem:key:VarInf}.

\pagebreak

\section{Crossing probabilities for quenched Voronoi percolation: weak bounds}\label{sec:RSW}

In this section we will prove a weaker version of Theorem~\ref{thm:RSW} for quenched Voronoi percolation in the plane, and deduce the corresponding theorem when $\eta$ is a subset of the half-plane. Crucially, these results will be sufficient to deduce a bound on the probability of the one-arm event that is strong enough for our application of the Schramm-Steif method in Section~\ref{sec:proof}.

\subsection{Crossing probabilities in the plane}\label{sec:RSW:plane}

In this section, $\eta$ will denote a Poisson process in the plane of intensity $1$, and $\omega \colon \eta \to \{-1,1\}$ will be a (uniform) random two-colouring of $\eta$.\footnote{Here $\omega$ is a colouring of $\eta$, which is countable. For ease of notation we use this convention, but remark that we are only interested in the restriction of $\omega$ to $\eta_R$, the subset of points whose cells intersect $R$, which is almost surely finite. We also continue to use $\Pr$ to denote the probability measure associated with choosing the pair $(\eta,\omega)$, and trust that this will cause the reader no confusion.} Recall that, given a rectangle $R$ with sides parallel to the axes, $H_R$ denotes the event that there is a red horizontal crossing of $R$ in the Voronoi tiling given by $\eta$, coloured by $\omega$. 

\begin{thm}\label{thm:RSW:plane}
For every rectangle $R \subset \RR^2$, there exists a constant $c > 0$, depending only on the aspect ratio of $R$, such that
$$\Pr\bigg( \Pr\big( H_R \,|\, \eta \big) \le \frac{1}{2^k} \bigg) \le \, (1-c)^k$$
for all sufficiently large $k \in \N$. 
\end{thm}

We begin by defining the following random variable, whose value depends on both $\eta$ and $\omega$, and which counts the maximum number of vertex-disjoint\footnote{By vertex-disjoint, we mean that no two of the crossings use points of the same Voronoi cell.} vertical crossings of $R$:
\begin{multline*}
X \, = \, X(\eta,\omega) \, := \, \max\big\{ m \in \N_0 \,: \, \textup{there exist $m$ vertex-disjoint, monochromatic}\\
\textup{vertical crossings $\{\gamma_1,\ldots,\gamma_m\}$ of $R$} \big\}.
\end{multline*}

The following lemma will be a key tool in our proof of Theorem~\ref{thm:RSW:plane}.

\begin{lemma}\label{magic:lemma}
For almost every $\eta$,
$$\Pr\big( H_R \,|\, \eta \big) \, = \, \Ex\big[ 2^{-X} \,|\, \eta \big].$$
\end{lemma}

As noted in the Introduction, we have found this lemma to be surprisingly powerful, and expect it to have many other applications. Despite this, the lemma is very easy to prove; indeed, it follows almost immediately from a basic (and well-known) fact about `colour-switching', see~\eqref{eq:colourswitching} below. For those readers who are unfamiliar with colour-switching, we will begin by giving a brief introduction. 

Consider the event that $X = k$, i.e., that there exist $k$, but not $k+1$, vertex-disjoint, monochromatic vertical crossings of $R$. The following standard algorithm provides a method of finding such paths. First, we discover the left-most monochromatic path (in the coloured Voronoi tiling of $R$ given by the pair $(\eta,\omega)$) starting from the left-most cell which intersects the lower side of the rectangle. If it reaches the top of the rectangle, then we add this monochromatic path to our collection; otherwise we discover the whole monochromatic component of our starting-point. In either case, we then discover the left-most monochromatic path entirely to the right of the already-discovered cells, starting from the next available cell on the lower side of $R$ (if it exists). Repeating this process until we reach the right-side of $R$, we obtain a collection $(\gamma_1,\ldots,\gamma_k)$ of disjoint monochromatic crossings.

An important feature of the algorithm above is that it allows us to use so-called `colour-switching' arguments, see e.g.~\cite{BN,Sm}. The crucial observation is that, for a given $\eta$, and a given collection of paths $(\gamma_1,\ldots,\gamma_k)$ obtained via the algorithm, there is a bijection between the set of colourings in which $\gamma_j$ is red, and those in which it is blue. Indeed, if we swap the colours of all cells that are on or to the right of $\gamma_j$, then the algorithm produces exactly the same set of paths. More generally, writing $\Pi \in \{-1,1\}^{k}$ for the sequence of colours of the paths $(\gamma_1,\ldots,\gamma_k)$, we have the following simple fact. For every $\sigma \in \{-1,1\}^k$, we have
\begin{equation}\label{eq:colourswitching}
\Pr\big( \Pi = \sigma \,\big|\, X = k, \, \eta \big) = \frac{1}{2^k}
\end{equation}
almost surely. Lemma~\ref{magic:lemma} is an almost immediate consequence of~\eqref{eq:colourswitching}.

\begin{proof}[Proof of Lemma~\ref{magic:lemma}]
Observe that the event $H_R$ holds if and only if all monochromatic vertical paths are red. By~\eqref{eq:colourswitching}, it follows that
\begin{align*}
\Pr\big( H_R \,|\, \eta \big) & \,=\, \sum_{k = 0}^\infty \Pr\Big( \Pi = (1,\ldots,1) \,\big|\, X = k, \, \eta \Big) \Pr\Big( X = k \,\big|\, \eta \Big)\\
& \,=\, \sum_{k = 0}^\infty \frac{ \Pr(X = k \, | \, \eta )}{2^k}\, = \, \Ex\big[ 2^{-X} \,|\, \eta \big],
\end{align*}
as required.
\end{proof}

In order to conclude to obtain the proof of Theorem~\ref{thm:RSW:plane}, we only need to show that $X$ cannot be too large. This will be a consequence of three properties of annealed Voronoi percolation: the FKG and BK inequalities, and the box crossing property, proved recently by the fourth author~\cite{T}. An event $A$ that depends on the pair $(\eta,\omega)$ is said to be \emph{red-increasing} if removing a blue point or adding a red point cannot cause the status of $A$ to change from true to false. For a proof of the following lemma, see~\cite[Lemma~8.14]{BRbook}. 

\begin{lemma}[The FKG inequality for annealed Voronoi percolation]\label{FKG}
Let $A$ and $B$ be red-increasing events. Then
$$\Pr(A \cap B ) \, \ge \, \Pr(A) \cdot \Pr(B),$$
and moreover the reverse inequality holds if $B$ is replaced by its complement $B^c$.
\end{lemma} 
 
The following lemma was proved by van den Berg~\cite{B}; his proof, sketched below, can also be found in the PhD thesis of Joosten~\cite[Section~4.3]{J} (who also refers to van den Berg). The corresponding inequality in the discrete setting is a celebrated result of van den Berg and Kesten~\cite{BK}.
 
\begin{lemma}[The BK inequality for annealed Voronoi percolation]\label{lem:BK}
Let $A$ and $B$ be red-increasing events. Then
$$\Pr(A \circ B) \, \le \, \Pr(A) \cdot \Pr(B),$$
where $A \circ B$ denotes the event that $A$ and $B$ occur disjointly.\footnote{For general events one must be a little careful in defining disjoint occurrence, but for events involving crossings the definition is straightforward: the crossings must be vertex-disjoint.}
\end{lemma}

\begin{proof}
Let $\ol{B}$ denote the event corresponding to $B$ with colours reversed. We have
\begin{align*}
\Pr(A \circ B) & \, = \, \Ex\big[ \Pr( A \circ B \,|\, \eta) \big] \, \le \, \Ex\big[ \Pr( A \cap \ol{B} \,|\, \eta) \big]\\
& \, = \, \Pr( A \cap \ol{B} ) \, \le \, \Pr(A) \Pr(\ol{B}) \, = \, \Pr(A) \Pr(B),
\end{align*}
where the first inequality follows from Reimer's inequality~\cite[Theorem~1.2]{Reimer}, and the second inequality follows from Lemma~\ref{FKG}.
\end{proof}

Our final tool is the following result of Tassion~\cite{T}. 

\begin{lemma}[The box-crossing property for annealed Voronoi percolation]\label{lem:Vincent:BXP}
There exists a constant $c_0 > 0$, depending only on the aspect ratio of $R$, such that
$$\Pr\big( H_R  \big) > c_0.$$
\end{lemma}

We are now ready to prove the weak box-crossing property for quenched Voronoi percolation in the plane.


\begin{proof}[Proof of Theorem~\ref{thm:RSW:plane}]
Given the pair $(\eta,\omega)$, define $X^+$ (resp.\@ $X^-$) to be the maximal number of disjoint red (resp.\@ blue) paths from top to bottom in $R$, so in particular $X = X^+ + X^-$. By Lemma~\ref{lem:Vincent:BXP} and the BK inequality for annealed Voronoi percolation, there exists a constant $c_0 > 0$, depending only on the aspect ratio of $R$, such that 
$$\Pr\big( X^+ \ge k \big) \, \le \, (1-c_0)^k$$
for all $k \ge 0$, and similarly $\Pr( X^- \ge k ) \le (1-c_0)^k$. Since the events $\big\{ X^+ \ge i \big\}$ and $\big\{ X^- < j \big\}$ are both red-increasing, it follows by the FKG inequality that
\begin{align}\label{eq:6}
\Pr\big(X\ge k\big) & \, \le \, \sum_{i+j\ge k}  \Pr\Big( \big\{ X^+ \ge i \big\} \cap \big\{ X^-\ge j \big\} \Big) \nonumber \\
& \, \le \, \sum_{i+j\ge k} (1-c_0)^{i+j} \, \le \, \frac{(1 - c)^k}{2}
\end{align}
for some constant $c > 0$, if $k$ is sufficiently large. Now, by Lemma~\ref{magic:lemma} we have
$$\Pr\big( H_R \,|\, \eta \big) \, = \, \Ex\big[ 2^{-X} \,|\, \eta \big] \, \ge \, \left(\frac12\right)^{k-1}\Pr\big( X < k \,|\, \eta \big),$$
almost surely. Thus, by Markov's inequality and~\eqref{eq:6}, we obtain
$$\Pr\bigg( \Pr\big( H_R \,|\, \eta \big) \le \left( \frac12\right)^{k} \bigg) \, \le \, \Pr\bigg( \Pr\big( X \ge k \,|\, \eta \big) \ge \frac12 \bigg) \, \le \, (1 - c)^k,$$
as required.
\end{proof}

\subsection{Quenched crossing probabilities in the half-plane}\label{sec:RSW:halfplane}

In order to bound one-arm events starting at points near the boundary of $R$, we will require a result analogous to Theorem~\ref{thm:RSW:plane} for a Poisson point process $\eta$ of intensity $1$ restricted to the half-plane
$$\HH \, := \, \big\{ (x,y) \in \RR^2 : x \ge 0 \big\}.$$ 
For each rectangle $R \subset \HH$ with sides parallel to the axes, let $H^*_R$ denote the event that there is a red horizontal crossing of $R$ in the Voronoi tiling of $\HH$ given by $\HH \cap \eta$, coloured by $\omega$.

\begin{thm}\label{thm:RSW:halfplane}
For every rectangle $R \subset \HH$, there exists a constant $c > 0$, depending only on the aspect ratio of $R$, such that 
$$\Pr\bigg( \Pr\big( H^*_R \,|\, \eta \big) \le \frac{1}{2^k} \bigg) \le \, (1-c)^k,$$
for all sufficiently large $k \in \N$. 
\end{thm}

The proof of Theorem~\ref{thm:RSW:halfplane} is identical to that of Theorem~\ref{thm:RSW:plane}, except we will need to replace Lemma~\ref{lem:Vincent:BXP} with the following bound in the annealed setting. 

\begin{lemma}\label{lem:annealedBXP:halfplane}
For every rectangle $R$, there exists a constant $c_1 > 0$, depending only on the aspect ratio of $R$, such that
$$\Pr\big( H^*_R \big) > c_1.$$
\end{lemma}

When the rectangle $R$ is sufficiently far from the boundary, Lemma~\ref{lem:annealedBXP:halfplane} follows from Lemma~\ref{lem:Vincent:BXP}, since the Voronoi tilings of $R$ is (with high probability) the same in both cases. We begin with an easy lemma that makes this statement precise.

\begin{lemma}\label{lem:boundary}
Given $\lambda > 0$, let $L > 0$ be sufficiently large. Let $\eta$ be a Poisson point process in the plane of intensity 1, and let $R$ be a $\lambda L \times L$ rectangle with
$$\min\{ x : (x,y) \in R \} \ge (\log L)^{2/3}.$$
Then the Voronoi tilings of $R$ induced by $\eta$ and by $\HH \cap \eta$ are non-identical with probability at most $1/L^3$.
\end{lemma}

\begin{proof}
If the Voronoi tilings of $R$ induced by $\eta$ and by $\HH \cap \eta$ are non-identical, then there must be a point $u \in R$ such that the closest point of $\eta$ to $u$ lies outside $\HH$. This implies that there is an empty ball of radius $(\log L)^{2/3}$ centred in $R$, the probability of which (by standard properties of Poisson processes) is super-polynomially small in $L$.
\end{proof}

We will assume from now on that $R$ is a $\lambda L \times L$
rectangle with $L$ sufficiently large, and such that the left-hand
side of $R$ is on the line $x = 0$. (Note that, by choosing $c_1$
sufficiently small, this may be assumed without loss of generality.)
The idea of the proof of Lemma~\ref{lem:annealedBXP:halfplane} is as
follows. Set $\ell = (\log L)^{2/3}$, and partition\footnote{We can
  deal with rounding issues by increasing $\ell$ slightly if
  necessary.} the set $\{ (x,y) \in R : 0 \le x \le \ell \}$ into
$L/\ell$ squares $S_1,\ldots,S_{L/\ell}$. Now observe that if there is
a red horizontal crossing of the rectangle $R' = \{ (x,y) \in R : x
\ge \ell \}$, but no red horizontal crossing of $R$, then there must
exist a square $S_j$ such that the following `3-arm event' holds (see
Figure~1, below). 

\begin{figure}[htbp]
  \centering
  \includegraphics[width=5cm]{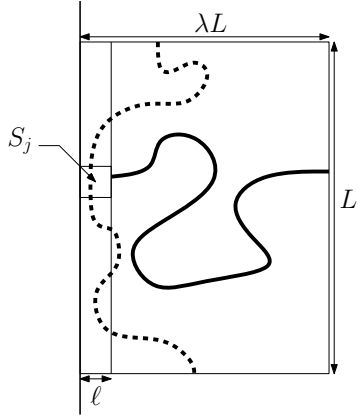}
  \caption{If there exists a red horizontal crossing in $R'$, but no
    red horizontal crossing of $R$, then a `3-arm event' must hold from
    a square $S_j$. (Red and blue paths are represented by solid and
    dotted paths, respectively.) }
  \label{fig:A3}
\end{figure}

\begin{defn}\label{def:eventA3}
For each $1 \le j \le L/\ell$, let $A^{(3)}(j)$
denote the event that the following hold:
\begin{itemize}
  \item[$(a)$] there is a red path from $S_j$ to the right-hand side of $R$ that is contained in $R'$.
  \item[$(b)$] there are blue paths from $S_j$ to the top and the bottom of $R$ that are contained in~$R$.
\end{itemize}
\end{defn}


By Lemma~\ref{lem:boundary} and the observations above, we have 
\begin{equation}\label{eq:3}
\Pr\big(H^*_{R}\big) \,\ge\, \Pr\big(H_{R'}\big) - \frac{1}{L^3} -\, \sum_{j=1}^{L/\ell} \Pr\big( A^{(3)}(j) \big).
\end{equation}
Thus, by Lemma~\ref{lem:Vincent:BXP}, it will suffice to prove the following lemma.

\begin{lemma}\label{lem:3arm}
There exists $c > 0$ such that
$$\sum_{j=1}^{L/\ell} \Pr\big( A^{(3)}(j) \big) \le L^{-c}$$
for all sufficiently large $L > 0$.
\end{lemma}

In the proof of Lemma~\ref{lem:3arm}, we will use the following two events: 
\begin{itemize}
\item $A^{(1)}(j)$ denotes the event that there is a red path from $S_j$ to the right-hand side of $R$ that is completely contained inside $R'$.
\item $A^{(2)}(j)$ denotes the event that there is a red path from $S_j$ to the top of $R$, and a blue path from $S_j$ to the bottom of $R$, both of which are completely contained inside~$R$. 
\end{itemize}
We are now ready to prove Lemma~\ref{lem:3arm}.

\begin{proof}[Proof of Lemma~\ref{lem:3arm}]
We first claim that, for each $j \in [L/\ell]$, we have
\begin{equation}\label{eq:A3vsA1A2}
\Pr\big( A^{(3)}(j) \big) \, \le \, \Pr\big( A^{(1)}(j) \big) \cdot \Pr\big( A^{(2)}(j) \big).\end{equation}
To prove~\eqref{eq:A3vsA1A2}, we apply colour-switching and the BK inequality for annealed Voronoi percolation (Lemma~\ref{lem:BK}). Indeed, the three paths in Definition~\ref{def:eventA3} are of alternating colours, so must be vertex-disjoint. Moreover, if such vertex-disjoint, monochromatic paths exist, then by colour-switching they are each equally likely to be red or blue.\footnote{To be slightly more precise, if such (vertex-disjoint, monochromatic) paths exist then we may choose the `left-most' such triple $(\gamma_1,\gamma_2,\gamma_3)$, and prove a result corresponding to~\eqref{eq:colourswitching} by switching the colours of all cells that are on or to the right of $\gamma_j$ for each $j \in \{1,2,3\}$. The left-most triple is obtained by choosing $\gamma_j$ to be the left-most monochromatic path to the top/right/bottom of $R$ that is entirely to the right of $\gamma_{j-1}$.}  Thus, letting $B^{(2)}(j)$ denote the event that there are vertex-disjoint red paths from $S_j$ to the top and bottom of $R$, both of which are completely contained inside $R$, we have
$$\Pr\big( A^{(3)}(j) \,|\, \eta \big) \, = \, \Pr\big( A^{(1)}(j) \circ B^{(2)}(j) \,|\, \eta \big) \quad \textup{and} \quad \Pr\big( A^{(2)}(j) \,|\, \eta \big) \, = \, \Pr\big( B^{(2)}(j) \,|\, \eta \big)$$
for almost all $\eta$. Taking expectation over $\eta$, and noting that $A^{(1)}(j)$ and $B^{(2)}(j)$ are both red-increasing events, we have 
$$\Pr\big( A^{(3)}(j) \big) \, \le \, \Pr\big( A^{(1)}(j) \big) \cdot \Pr\big( B^{(2)}(j) \big) \, = \, \Pr\big( A^{(1)}(j) \big) \cdot \Pr\big( A^{(2)}(j) \big)$$
by the BK inequality, as claimed. 

By~\eqref{eq:A3vsA1A2}, the lemma is an immediately consequence of the following two claims. 

\medskip
\noindent \textbf{Claim 1:} $\Pr\big( A^{(1)}(j) \big) \le L^{-2c}$ for some constant $c > 0$. 

\begin{proof}[Proof of claim]
By Lemma~\ref{lem:boundary}, the probability that the Voronoi tiling of $R'$ by $\HH \cap \eta$ differs from that by $\eta$ is at most $1/L^3$. The claim therefore follows from the corresponding statement in the plane, which is a standard consequence of Lemma~\ref{lem:Vincent:BXP}, see~\cite[Theorem~3]{T}.
\end{proof}
  
\medskip
\noindent \textbf{Claim 2:} $\ds\sum_{j = 1}^{L/\ell}  \Pr\big( A^{(2)}(j) \big) \le 2^{C\ell}$ for some constant $C > 0$. 

\begin{proof}[Proof of claim]
For each $j \in [L/\ell]$, choose a (distinct) cell $u(j) \in \eta$ which intersects both $S_j$ and the left-hand side of $R$. Consider the event $E(j)$, illustrated in Figure~2, 
that there is a red path from $u(j)$ to the top of $R$, and a blue path from the cell $u'(j)$ immediately below $u(j)$ to the bottom of $R$, both of which are completely contained inside $R$. 

\begin{figure}[htbp]
  \label{fig:Ej}
  \centering
  \includegraphics[width=5cm]{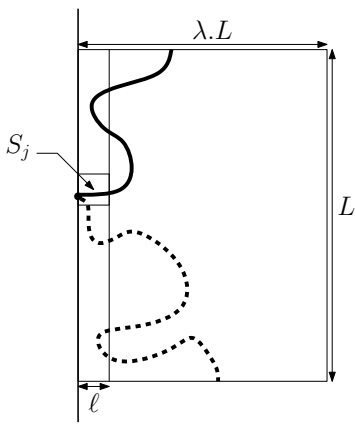}
  \caption{Illustration of the event $E(j)$: there exist a red path from a cell $u(j)$ intersecting $S_j$ and the left boundary of $R$ to the top of $R$, and a blue path from the cell $u'(j)$ immediately below $u(j)$ to the bottom of $R$.}
\end{figure}

We claim that
\begin{equation}\label{eq:10}
    \Pr\big( A^{(2)}(j)\big) \le \, 2^{C\ell} \cdot \Pr\big( E(j)\big)
\end{equation}
for each $j \in [L/\ell]$. To prove this, we simply use brute force to tunnel from the boundary of $S_j$ to $u(j)$. Indeed, as long as, for every pair of points $\{v,w\}$ on the boundary of $S_j$, there are vertex-disjoint paths (in the Voronoi tiling of $S_j$) from $u(j)$ to $v$ and from $u'(j)$ to $w$, each of length at most $C\ell/2$, then we can connect $u(j)$ and $u'(j)$ to the endpoints of the paths guaranteed by the event $A^{(2)}(j)$ with probability at least $2^{-C\ell}$. But such paths exist unless the Poisson process is much denser in $S_j$ than one would expect, and this occurs with probability that is super-polynomially small in $L$.

Finally, simply note that, since at most one of the events $E(j)$ can occur, we have
$$\sum_{j = 1}^{L/\ell} \Pr\big( A^{(2)}(j)\big) \le \, 2^{C\ell} \sum_{j = 1}^{L/\ell} \Pr\big( E(j) \big) \le \, 2^{C\ell}$$
as claimed.
\end{proof}

Combining Claims~1 and~2 with~\eqref{eq:A3vsA1A2}, and recalling that $\ell = (\log L)^{2/3}$, we obtain 
$$\sum_{j=1}^{L/\ell} \Pr\big( A^{(3)}(j) \big) \, \le \, \sum_{j=1}^{L/\ell} \Pr\big( A^{(1)}(j) \big) \cdot \Pr\big( A^{(2)}(j) \big) \, \le \, L^{-2c} \cdot 2^{C\ell} \, \le \, L^{-c}$$
if $L$ is sufficiently large, as required.
\end{proof}

For completeness, let us spell out how to deduce Lemma~\ref{lem:annealedBXP:halfplane} from Lemmas~\ref{lem:Vincent:BXP} and~\ref{lem:3arm}. 

\begin{proof}[Proof of Lemma~\ref{lem:annealedBXP:halfplane}]
By~\eqref{eq:3}, and Lemmas~\ref{lem:Vincent:BXP} and~\ref{lem:3arm}, we have 
$$\Pr\big(H^*_{R}\big) \,\ge\, \Pr\big(H_{R'}\big) - \frac{1}{L^3} -\, \sum_{j=1}^{L/\ell} \Pr\big( A^{(3)}(j) \big) \,\ge\, c_0 - \frac{1}{L^3} - L^{-c} \, \ge \, c_1$$
if $L$ is sufficiently large, as required.
\end{proof}

As noted above, Theorem~\ref{thm:RSW:halfplane} follows by repeating the proof of Theorem~\ref{thm:RSW:plane}, using Lemma~\ref{lem:annealedBXP:halfplane} in place of Lemma~\ref{lem:Vincent:BXP}.

\subsection{A bound on the probability of the 1-arm event in a rectangle}

To finish the section, let us deduce the following proposition from Theorems~\ref{thm:RSW:plane} and~\ref{thm:RSW:halfplane}. Let $R \subset \RR^2$ be a rectangle of area $n$, and let $\eta \subset R$ be a set of $n$ points, each chosen uniformly at random, and let $\omega$ be a uniform colouring of $\eta$. (Thus, the distribution of $\eta$ inside $R$ is very close to that of a Poisson process of intensity 1.) Given $u \in R$ and $d > 0$, we write $M(u,d)$ for the event (depending on $\eta$ and $\omega$) that there is a monochromatic path from $u$ to some point of $R$ at $\ell_2$-distance $d$ from $u$.

We will use the following result in Section~\ref{sec:proof} in order to bound the `revealment' of our randomized algorithm, and hence to deduce our main theorems. Since the deduction of this result from the box-crossing property is standard, we will be fairly brief with the details.

\begin{prop}\label{prop:quenched:onearm}
For every $\gamma > 0$, there exists $\eps > 0$ such that the following holds. Suppose that $d = d(n) \to \infty$ as $n \to \infty$ and let $u \in R$. Then 
$$\Pr\Big( \Pr\big( M(u,d) \,|\, \eta \big) \ge d^{-\eps} \Big) \, \le \, d^{-\gamma}$$
for all sufficiently large $n \in \N$.
\end{prop}

\begin{proof}
Fix a point $u \in R$, and for each $j \in \N$, let $\AA_j$ denote the square annulus, centred on $u$, with inner side-length $7^j$ and outer side-length $3 \cdot 7^j$. Let $O_j$ denote the event that there is a blue circuit around the (partial) annulus $\AA_j \cap R$. (This circuit must either be closed, or have both its endpoints on the boundary of $R$; in either case, it must separate $u$ from the exterior of $\AA_j$.) Let
$$J = \big\{ j \in \N \,:\, \sqrt{d} \le 7^{j+1} \le d \big\},$$
and consider the collection of annuli $\C = \{ \AA_j : j \in J\}$. Roughly speaking, we will show that either at least half the annuli in $\C$ contain an `unusual' collection of points, or the probability that none of the events $O_j$ occurs is at most $d^{-\eps}$.

Let $c$ be the constant in Theorems~\ref{thm:RSW:plane} and~\ref{thm:RSW:halfplane}, and fix a large constant $k \in \N$ (depending on~$c$ and~$\gamma$). For each $j \in J$, let $D^{(1)}_j$ denote the event (depending on $\eta$) that $\Pr\big( O_j \,|\, \eta \big) > 2^{-4k}$, and let $D^{(2)}_j$ denote the event that for every $z \in \AA_j$, there exists some point $x \in \eta$ at distance at most $\log d$ from $z$. Define
$$D_j \, := \, D^{(1)}_j \cap D^{(2)}_j,$$
and observe that the events $D_j$ are independent. We claim that
\begin{equation}\label{eq:Djbound}
\Pr(D_j) \, \ge \, 1 - 5(1-c)^k
\end{equation}
if $d$ is sufficiently large. To prove~\eqref{eq:Djbound}, note first (cf. Lemma~\ref{lem:boundary}) that
$$\Pr\big( D^{(2)}_j \big) \, \to \, 1$$
as $d \to \infty$. Next, observe that, by the FKG inequality\footnote{Note that we are applying this in the quenched world, so the usual FKG inequality (also known as Harris' lemma) is sufficient.} and Theorems~\ref{thm:RSW:plane} and~\ref{thm:RSW:halfplane}, we have
$$\Pr\Big( \Pr\big( O_j \,|\, \eta \big) \le 2^{-4k} \Big) \, \le \, 4(1-c)^k$$
for all sufficiently large $k \in \N$. Thus $\Pr(D_j^c) \le 4(1-c)^k + o(1) \le 5(1-c)^k$, as claimed. 

Now, let $D^*$ denote the event that $D_j$ holds for at least half of the elements $j \in J$, and observe that
$$\Pr(D^*) \, \ge \, 1 - 2^{|J|} \left( 5(1-c)^k \right)^{|J|/2} \ge \, 1 -  d^{-\gamma},$$
since $|J| = \Omega\big( \log d \big)$, and $k$ was chosen sufficiently large in terms of $\gamma$ and $c$. But for those $\eta$ such that $D^*$ holds, we have
$$\Pr\big( M(u,d) \,|\, \eta \big) \, \le \, \Pr\bigg( \bigcap_{j \in J} O_j^c \,\big|\, \eta \bigg) \, = \, \prod_{j \in J} \Pr\big( O_j^c \,|\, \eta \big) \, \le \, \big( 1 - 2^{-4k} \big)^{|J|/2} \, \le \, d^{-\eps},$$
for some $\eps > 0$, as required. 
\end{proof}

\section{The proof of the main theorems}\label{sec:proof}

In this section we will use the algorithm method of Benjamini, Kalai and Schramm~\cite{BKS} and Schramm and Steif~\cite{SS}, together with the results proved in the previous sections, in order to bound the sum of the squares of the influences, and hence deduce the following theorem. Throughout this section, let us fix a rectangle $R \subset \RR^2$. 


\begin{thm}\label{thm:BKS:rectangle}
There exists $c = c(R) > 0$ such that the following holds. Let $\eta \subset R$ be a set of~$n$ points, each chosen uniformly at random, and let $\omega \colon \eta \to \{-1,1\}$ be a uniform colouring. Then
$$\Pr\Big( \big| \Pr( H_R \,|\, \eta ) - \Pr( H_R ) \big| \ge n^{-c} \Big) \le n^{-c}$$
for all sufficiently large $n \in \N$.
\end{thm}

Theorem~\ref{thm:BKSconj} follows immediately from Theorem~\ref{thm:BKS:rectangle} by taking $R = [0,1]^2$, and recalling that in this case $\Pr(H_R) = 1/2$. Note that Theorem~\ref{thm:RSW} also follows from Theorem~\ref{thm:BKS:rectangle}, together with the RSW theorem for annealed Voronoi percolation, which implies that $\Pr(H_R)$ is bounded away from 0 and 1 uniformly in $n$.

The randomised algorithm method was introduced in~\cite{SS} and used there to show that the are exceptional times in dynamical site percolation on the triangular lattice. Since this method is well-known, and our application is rather standard, we shall be somewhat less careful with the details than in earlier sections, focusing instead on conveying the main ideas. Given a randomized algorithm $\A$ that determines a function $f_n \colon \{-1,1\}^n \to \{0,1\}$, define the revealment of $\A$ to be
$$\delta_\A(f_n) \, := \, \max_{i \in [n]} \Pr\big( i \textup{ is queried by } \A \big).$$
Schramm and Steif~\cite{SS} proved a very powerful bound on the Fourier coefficients of a real-valued function on the hypercube $\{-1,1\}^n$ in terms of the revealment of any randomized algorithm that determines $f$. For monotone Boolean functions their result easily implies the following theorem, which will be sufficient for our purposes.


\begin{thm}[Schramm and Steif, 2010]\label{thm:SS}
Given a monotone function $f \colon \{-1,1\}^n \to \{0,1\}$, and a randomized algorithm $\A$ that determines $f$, we have
$$\sum_{m = 1}^n \Inf_m(f)^2 \, \le \, \delta_\A(f).$$
\end{thm}


By Theorems~\ref{thm:VarInf} and~\ref{thm:SS}, it only remains to show that, with probability at least $1 - n^{-c}$, there exists an algorithm that determines $f_R^\eta$ with revealment that is polynomially small in $n$. (Recall that $f_R^\eta \colon \{-1,1\}^\eta \to \{0,1\}$ denotes the function such that $f_R^\eta(\omega) = 1$ if and only if $H_R$ occurs.) The algorithm we will use is essentially the same as that introduced in~\cite{SS}, so we shall describe it in an intuitive (and therefore slightly non-rigorous) fashion, and refer the reader to the original paper for the details.

\begin{SSalg}
Let $\A$ be the algorithm that, given $\eta$, queries bits of $\omega$ as follows:
\begin{itemize}
\item[1.] Choose a point $x$ in the middle third of the left-hand side of $R$ uniformly at random. 
\item[2.] Explore the boundary between red and blue cells, with red on the left, starting from $x$. Here we place boundary conditions as follows: the left-hand side of $R$ is red above $x$, and blue below, and the bottom of $R$ is also blue. If this path:
\begin{itemize}
\item[$(a)$] reaches the right-hand side of $R$, then $f_R^\eta(\omega) = 1$.
\item[$(b)$] reaches the bottom of $R$, and ends at the top, then $f_R^\eta(\omega) = 0$.
\item[$(c)$] ends at the top of $R$ without reaching the bottom, then go to step 3.
\end{itemize}
\item[3.] Explore the boundary between red and blue cells, with red on the right, starting from $x$. Here we reverse the boundary conditions, i.e., the left-hand side of $R$ is blue above $x$, and red below, and the top of $R$ is also blue. If this path:
\begin{itemize}
\item[$(a)$] reaches the right-hand side of $R$, then $f_R^\eta(\omega) = 1$.
\item[$(b)$] otherwise $f_R^\eta(\omega) = 0$.
\end{itemize}
\end{itemize}
\end{SSalg}

Note that we only query bits as needed, i.e., we query those vertices whose cell we meet along one of our paths. Still following~\cite{SS}, this allows us to immediately bound the revealment of $\A$ as follows. Recall that, given $u \in R$ and $d > 0$, we write $M(u,d)$ for the event that there is a monochromatic path from $u$ to some point of $R$ at $\ell_2$-distance $d$ from $u$.
 
\begin{lemma}\label{lem:SS}
Let $\A$ be the Schramm-Steif randomized algorithm. Then
$$\delta_\A(f_R^\eta) \, \le \, \max_{u \in \eta} \Pr\Big( M\big( u, n^{-1/4} \big) \,\big|\, \eta \Big) + O\big( n^{-1/4} \big)$$
almost surely.
\end{lemma}

\begin{proof}[Sketch of the proof]
Let $u \in \eta$, and consider the probability that $u$ is queried by $\A$. First, note that the probability that the random start-point $x$ is within distance $n^{-1/4}$ of $u$ is $O(n^{-1/4})$. But if the distance between $u$ and $x$ is greater than $n^{-1/4}$, and $u$ is nonetheless queried by $\A$, then the event $M\big( u, n^{-1/4} \big)$ holds.
\end{proof}

To bound the revealment of $\A$, it will therefore suffice to bound the probability of the event $M\big( u, n^{-1/4} \big)$. This is an immediate consequence of Proposition~\ref{prop:quenched:onearm}.

\begin{lemma}\label{lem:max1arm}
For every $\gamma > 0$, there exists $c > 0$ such that
$$\Pr\Big( \max_{u \in \eta} \Pr\Big( M\big( u, n^{-1/4} \big) \,\big|\, \eta \Big) \ge n^{-c} \Big) \, \le \, \frac{1}{n^\gamma}$$
for all sufficiently large $n \in \N$.
\end{lemma}

\begin{proof}
Renormalizing $R$ to have area $n$, we find that distances are multiplied by $\Theta\big( \sqrt{n} \big)$, and so we may apply Proposition~\ref{prop:quenched:onearm} with $d = \Theta\big( n^{1/4} \big)$. The claimed bound now follows from the proposition, using the union bound over points $u \in \eta$.
\end{proof}

We are now ready to deduce our bound on the sum of the squares of the influences, and hence (by Theorem~\ref{thm:VarInf} and the results of~\cite{BKS} and~\cite{T}) our main theorems.

\begin{lemma}\label{lem:influences}
For every $\gamma > 0$, there exists $c > 0$ such that
$$\Pr\bigg( \sum_{m = 1}^n \Inf_m\big( f_R^\eta \big)^2 \ge n^{-c} \bigg) \, \le \, \frac{1}{n^\gamma}$$
for all sufficiently large $n \in \N$.
\end{lemma}

\begin{proof}
This follows immediately from Theorem~\ref{thm:SS}, together with Lemmas~\ref{lem:SS} and~\ref{lem:max1arm}. Indeed, we have
$$\sum_{m = 1}^n \Inf_m(f_R^\eta)^2 \, \le \, \delta_\A(f_R^\eta) \, \le \, n^{-c}$$
with probability at least $1 - n^{-\gamma}$, as required.
\end{proof}

We can now deduce Theorem~\ref{thm:BKS:rectangle}.

\begin{proof}[Proof of Theorem~\ref{thm:BKS:rectangle}]
By Corollary~\ref{cor:VarInf} and Lemma~\ref{lem:influences} (applied with $\gamma = 1$, say), we have
$$\Pr\Big( \big| \Pr( H_R \,|\, \eta ) - \Pr( H_R ) \big| \ge n^{-c} \Big) \, \le \, n^{-c}$$
for some $c > 0$, and all sufficiently large $n \in \N$, as required.
\end{proof}

Finally, let us note that the theorems stated in the Introduction all follow easily.

\begin{proof}[Proof of Theorem~\ref{thm:BKSconj}]
As noted above, this is an immediate corollary of Theorem~\ref{thm:BKS:rectangle}. Indeed, simply set $R = [0,1]^2$ and observe that $\Pr(H_R) = 1/2$.
\end{proof}

Next we deduce that quenched Voronoi percolation is noise sensitive at criticality. We remark that, although we do not give the details here, it is a standard consequence of the method of~\cite{SS} that a stronger statement holds: That~\eqref{eq:NSdef} holds even if $\eps = n^{-c}$ for some (sufficiently small) constant $c > 0$.

\begin{proof}[Proof of Theorem~\ref{thm:QV:NS}]
The theorem follows immediately from Lemma~\ref{lem:influences}, together with the Benjamini-Kalai-Schramm Theorem (see~\eqref{eq:BKSthm}).
\end{proof}

Finally, we prove the quenched box-crossing property for Voronoi percolation.

\begin{proof}[Proof of Theorem~\ref{thm:RSW}]
This follows from Theorem~\ref{thm:BKS:rectangle}, together with the RSW theorem for annealed Voronoi percolation, which was proved in~\cite{T}. Indeed, for every rectangle $R$, there exists $c_0 > 0$ such that
$$c_0 \, < \, \Pr(H_R) \, < \, 1 - c_0,$$
see~\cite[Theorem~3]{T}. Hence, by Theorem~\ref{thm:BKS:rectangle}, we have, for some $c > 0$,  
\begin{equation}\label{eq:lastequation}
\Pr\Big( c < \Pr\big( H_R \,|\, \eta \big) < 1 - c \Big) \, \ge \, 1 - n^{-c}
\end{equation}
for all sufficiently large $n \in \N$.
\end{proof}

Note that, by partitioning $R$ into a bounded number of strips and taking $c$ sufficiently small, inequality~\eqref{eq:lastequation} implies that the probability of the event $\Pr\big( H_R \,|\, \eta \big) \not\in (c,1-c)$ can be made smaller than any given polynomial, as claimed.

\section*{Acknowledgements} 

The authors would like to thank Rob van den Berg for allowing us to include his proof of the BK inequality for annealed Voronoi percolation. The first author would also like to thank Elchanan Mossel and G\'abor Pete for encouraging discussions, and the second and third authors would like to thank Paul Balister and B\'ela Bollob\'as for a number of very interesting conversations about quenched Voronoi percolation.

\end{document}